\documentclass[copyright,creativecommons]{eptcs}
\usepackage{microtype}
\title{On Natural Deduction for Herbrand Constructive Logics III: The Strange Case of the Intuitionistic Logic of Constant Domains}
\author{Federico Aschieri\thanks{Funded by the Austrian Science Fund FWF START grant Y 544--N23}\\ Institut f\"ur Logic and Computation\\
Technische Universit\"at Wien }
\def\authorrunning{F. Aschieri} 
\def\titlerunning{The Intuitionistic Logic of Constant Domains}

\usepackage{amsmath}
\usepackage{bussproofs}
\usepackage{amsthm}
\usepackage{amssymb}
\usepackage[inline]{enumitem}
\usepackage{hyperref}
\usepackage{cleveref}
\usepackage{upgreek}

\providecommand{\urlalt}[2]{\href{#1}{#2}}
\providecommand{\doi}[1]{doi:\urlalt{http://dx.doi.org/#1}{#1}}
\providecommand{\urlalt}[2]{\href{#1}{#2}}
\providecommand{\arxiv}[1]{ArXiv:\urlalt{https://arxiv.org/abs/#1}{#1}}
\newcommand{\Language}                 {\mathcal{L}}
\newcommand{\pair}[2]{\langle #1, #2 \rangle}
\newcommand{\proj}[2]{#2 \pi_{#1}}
\newcommand{\inj}[2]{\upiota_{#1}(#2)}
\newcommand{\case}[3]{#1[#2, #3]}
\newcommand{\enc}[2]{(#1, #2)}
\newcommand{\dest}[3]{#1[(\alpha,#2).#3]}

\newcommand{\IL}{\mathsf{IL}}
\newcommand{\HA}{\mathsf{HA}}

\newcommand{\efq}[2]{{\mathsf{efq}_{#1}\, #2}}

\newcommand{\prop}[1]    {{\mathsf{#1}}}

\newcommand{\rec}{\mathsf{R}}
\newcommand{\suc}{\mathsf{S}}

\newcommand{\D}         {{\mathcal{D}}}
\newcommand{\CD}      {{\mathsf{CD}}}
\newcommand{\dum}   {{\mathsf{0}}}
\newcommand{\ILB}     {{\mathsf{IL}^{\bot}}}
\newcommand{\trans}[1] {{#1}^{*}}

\newcommand{\dlinea}{\leavevmode\hrule\vspace{1pt}\hrule\mbox{}}

\newtheorem{theorem}{Theorem}

\newtheorem{proposition}{Proposition}
\newtheorem{definition}{Definition}

\crefformat{section}{Section \S#2#1#3}

\begin{document}
\maketitle
 \begin{abstract}
 The logic of constant domains is intuitionistic logic extended with the so-called forall-shift axiom, a classically valid statement which implies the excluded middle over decidable formulas. Surprisingly, this logic is constructive and so far this has been proved by cut-elimination for ad-hoc sequent calculi. 
 Here we use the methods of natural deduction and Curry-Howard correspondence to provide a simple computational interpretation of the logic.
\end{abstract}

\section{Introduction}

This is the third in a series of papers about natural deduction for Herbrand constructive logics, which we defined to be intermediate logics satisfying Herbrand's Theorem for \emph{every} existential statement \cite{AschieriDummett, AschieriManighetti}. Indeed, intermediate logics prove intuitionistically as well as classically valid theorems, yet they often possess a strong constructive flavour. Our aim is to provide computational interpretations of them, using the Curry-Howard correspondence. Results about constructiveness of intermediate logics are scattered in the literature, are proved with diverse techniques and by means of a variety of logical systems. In contrast, natural deduction offers a systematic approach:   new axioms can be immediately translated into intuitive inference rules and one has only to provide computational readings of them.  This may require some ingenuity and devising termination proofs may be difficult, but the natural deduction framework offers a proven methodology that often leads to full success. The hope is to provide a general theory explaining why many intermediate logics are turning out to be Herbrand constructive.

With this general goal in mind, in this paper we give yet another example: it is the turn of the intuitionistic logic of constant domains $\CD$. It arises from a very natural Kripke-style semantics, which was proposed by Grzegorczyk \cite{Grzegorczyk} as philosophically plausible interpretation of intuitionistic logic. Technically, one only allows Kripke models with constant domains, that is, the set of individuals of the interpretation does not grow with the state of knowledge. This indeed is a natural conditions, which avoids some mathematical oddities of Kripke semantics, like the set of integers growing when interpreting intuitionistic Arithmetic \cite{Troelstra2}.

$\CD$ is obviously Herbrand constructive, because it is, in fact, fully constructive: if $\CD \vdash A\lor B$, then $\CD\vdash A$ or $\CD\vdash B$; if $\CD\vdash\exists\alpha\, A$, then $\CD\vdash A[m/\alpha]$ for some term $m$ of the language. Therefore, it must admit a natural computational interpretation. However, the constructiveness of $\CD$ was proved surprisingly late  by means of a rather complex sequent calculus \cite{Kashima}. Indeed, Fitting \cite{Fitting} described the problem of finding a \emph{simple} deductive system for $\CD$ as being  ``\emph{a mildly nagging problem for some time}''. Fitting did indeed propose a conceptually simple nested sequent calculus for $\CD$. However, the disjunction property just does not follow from the cut-elimination theorem, because the disjunction inference rule is the one from classical logic. 

$\CD$ indeed can be formalized as intuitionistic logic extended with a classically quite strong principle, the constant domain axiom:
$$\forall \alpha (A\lor B)\rightarrow \forall \alpha A\, \lor\, B (\mbox{ $\alpha$ not occurring in $B$})$$
As noticed by Fitting \cite{FittingMistake}, it implies the excluded middle over decidable formulas:
$$\forall \alpha (A(\alpha)\lor \lnot A(\alpha))\rightarrow \forall \alpha A(\alpha)\, \lor\, \exists \alpha \lnot A(\alpha)$$ 
This is rather puzzling. How can $\CD$ possibly be constructive? One expects $\CD$ to be Mr. Hyde and it turns out to be Dr. Jekyll.

 Our natural deduction and Curry-Howard interpretation for $\CD$ will provide a simple explanation. Namely, when proving a disjunction or an existential statement, the constant domain axiom fully complies the Brouwer-Heyting-Kolmogorov semantics: a proof of $\forall \alpha (A(\alpha)\lor B)$ does yield a proof of $\forall \alpha\, A(\alpha)$ or a proof of $B$. The reason, as we shall show, is the purely logical setting. The proof of $A(\alpha)\lor B$ cannot depend on $\alpha$, that is,  must be uniform. There is not much choice then: either there is a uniform proof of $A(\alpha)$ and thus of $\forall \alpha A(\alpha)$ or there is a proof of $B$. This in sharp contrast with actual mathematical proofs, for example in Arithmetic. A proof of $\forall \alpha (A(\alpha)\lor B)$, may very well deliver a proof of $A(\alpha)$ for some values of $\alpha$ and proof of $B$ for others. As a consequence, we shall show that as a result, Heyting Arithmetic \cite{Troelstra} extended with the constant domain axiom is not constructive.

At the current stage in this series of papers, we can already draw a philosophical lesson. \emph{Analiticity}, or better its formalization as \emph{Subformula Property},  is neither a sufficient nor a necessary condition for a proof system to be meaningful. Nested sequents \cite{Fitting} are simple and analytic, but do not deliver the existence and disjunction properties. The sequents in \cite{Kashima}  do deliver these properties, they are analytic, but not complicated. Our present and former natural deductions \cite{AschieriDummett, AschieriManighetti} do appear simple, do deliver the properties, but rather different and more elaborated ideas are needed to make them analytic.

 \subsection{Plan of the Paper}
 In \Cref{sec:ND}, we present our simple computational interpretation of $\CD$. We prove its constructivness and its strong normalization, by means of a simple translation into intuitionistic lambda terms. In \Cref{sec:arithmetic}, we focus on Heyting Arithmetic extended with the constant domain axiom and show it cannot be constructive.

\section{Natural Deduction and Curry-Howard for $\CD$}
\label{sec:ND}
In this section we describe a standard natural deduction system for intuitionistic first-order logic, with a term assignment based on the Curry-Howard correspondence (e.g. see \cite{Sorensen}), and add on top of it an operator which formalizes the constant domain axiom. We shall then immediately prove the Subject Reduction Theorem, stating that the reduction rules preserve the type and thus represent sound proof transformations. 
Afterwards, we give a simple proof of strong normalization, by encoding proof terms of $\CD$ into intuitionistic proof terms. Finally, we shall inspect normal forms and based on their forms, we shall conclude that $\CD$ is constructive.


We start with the standard first-order language of formulas

\begin{definition}[Language of $\CD$]\label{definition-languagear}
The language $\Language$ of $\CD$ is defined as follows.
\begin{enumerate}

\item
The \textbf{terms} of $\Language$ are inductively defined as either variables $\alpha, \beta,\ldots$ or constants  $\mathsf{c}$ or expressions of the form $\mathsf{f}(m_{1}, \ldots, m_{n})$, with $\mathsf{f}$ a function constant of arity $n$ and $m_{1}, \ldots, m_{n}\in\Language$.  

\item
There is a countable set of \textbf{predicate symbols}. The \emph{atomic formulas} of $\Language$ are all the expressions of the form $\mathcal{P}(m_{1}, \ldots, m_{n})$ such that  $\mathcal{P}$ is a predicate symbol of arity $n$ and $m_{1}, \ldots, m_{n}$ are terms of $\Language$. We assume to have a $0$-ary predicate symbol $\bot$ which represents falsity.

\item
The \textbf{formulas} of $\Language$ are built from atomic formulas of $\Language$ by the logical constants  $\lor,\land,\rightarrow, \forall,\exists$, with quantifiers ranging over  variables $\alpha, \beta, \ldots$: if $A, B$ are formulas, then $A\land B$, $A\lor B$, $A\rightarrow B$, $\forall \alpha\, A$, $\exists \alpha\, B$ are formulas. The  logical negation $\lnot A$ can be introduced, as usual, as a shorthand  for the formula $A\rightarrow\bot$.

\end{enumerate}

\end{definition}

Next, we consider the usual natural deduction system for intuitionistic first-order logic \cite{Prawitz, Sorensen}, to which we add a constant $\D^{I}: I$ for every instance $I$ of the constant domain axiom.  The resulting Curry-Howard system is called $\CD$ and is presented in \cref{fig:system}. 
As usual, the notation $\CD \vdash t: A$ stands for $\vdash t: A$ and means provability without assumptions.
The reduction rules for $\CD$ are presented in \cref{fig:red-ilmp} and include the ordinary ones of lambda calculus. The reduction rules for $\D$ are extremely simple as well and their task is just extracting what is already implicit in $\D$ argument: either a proof term for $\forall\alpha A$ or one for $B$.  As usual, we  omit types of variables whenever they do not matter; for arbitrary terms $t,u$, the relation $t\mapsto u$ holds whenever $u$ is obtained from $t$ by applying a reduction rule inside of $t$.


 \begin{figure*}[!htb]
 
\footnotesize{
\dlinea

\begin{description}



\item[Axioms] 
$\begin{array}{c}    x^A: A
\end{array}\ \ \ \ $
\\

\item[Conjunction] 
$\begin{array}{c}  u:  A\ \ \ \  t: B\\ \hline  \langle
u,t\rangle:
A\wedge B
\end{array}\ \ \ \ $
$\begin{array}{c}  u: A\wedge B\\ \hline u\,\pi_0: A
\end{array}\ \ \ \ $
$\begin{array}{c}   u: A\wedge B\\ \hline  u\,\pi_1 : B
\end{array}$
\vskip 0.15in
\item[Implication] 
$\begin{array}{c}   t: A\rightarrow B\ \ \  u:A \\ \hline
 t u:B
\end{array}\ \ \ \ $
 \AxiomC{$[x^{A}: A]$}
\noLine
\UnaryInfC{$\vdots$}
\noLine
\UnaryInfC{$u: B$}
\UnaryInfC{$\lambda x^{A} u: A\rightarrow B$}
\DisplayProof
\vskip 0.15in
\item[Disjunction Introduction] 
$\begin{array}{c} u: A\\ \hline  \inj{0}{u}: A\vee B
\end{array}\ \ \ \ $
$\begin{array}{c}   u: B\\ \hline \inj{1}{u}: A\vee B
\end{array}$
\vskip 0.15in

\item[Disjunction Elimination] 
\AxiomC{$u: A\lor B$}
\AxiomC{$[x^{A}: A]$}
\noLine
\UnaryInfC{$\vdots$}
\noLine
\UnaryInfC{$w_{1}: C$}
\AxiomC{$[y^{B}: B]$}
\noLine
\UnaryInfC{$\vdots$}
\noLine
\UnaryInfC{$w_{2}: C$}
\TrinaryInfC{$u\, [x^{A}.w_{1}, y^{B}.w_{2}]: C$}
\DisplayProof
\vskip 0.15in

\item[Universal Quantification] 
$\begin{array}{c} u:\forall \alpha\, A\\ \hline   u m: A[m/\alpha]
\end{array}\ \ \ $
$\begin{array}{c}   u: A\\ \hline  \lambda \alpha\, u:
\forall \alpha\, A
\end{array}$\\
\vskip 0.1in
where $m$ is any term of  the language $\Language$ and $\alpha$ does not occur
free in the type $B$ of any free variable  $x^{B}$ of $u$.\\
\vskip 0.1in

\item[Existential Quantification] 
$\begin{array}{c}  u: A[m/\alpha]\\ \hline  (
m,u):
\exists
\alpha\, A
\end{array}\ \ \ $
\AxiomC{$u: \exists \alpha\, A$}
\AxiomC{$[x^{A}: A]$}
\noLine
\UnaryInfC{$\vdots$}
\noLine
\UnaryInfC{$t: C$}
\BinaryInfC{$u\, [(\alpha, x^{A}). t]: C$}
\DisplayProof
\\\vskip 0.1in
where $\alpha$ is not free in $C$ nor in the type $B$ of any free variable of $t$.\\
\vskip 0.15in

\item[Constant Domain Axiom]
$\D^{\footnotesize{\forall \alpha (A\lor B)\rightarrow \forall \alpha A\, \lor\, B }} : \forall \alpha (A\lor B)\rightarrow \forall \alpha A\, \lor\, B  $,
\\\vskip 0.1in
 where $\alpha$ does not occur in $B$.
\vskip 0.15in
\item[Ex Falso Quodlibet] 
$\begin{array}{c}  \Gamma \vdash u: \bot \\ \hline \Gamma\vdash  \efq{P}{u}:
P
\end{array}$\\
with $P$ atomic.

\end{description}
}

\dlinea
\caption{Term Assignment Rules for $\CD$}\label{fig:system}
\end{figure*}

\begin{figure*}[!htb]
\dlinea

\begin{description}

 \item[Reduction Rules for $\CD$]
 \[(\lambda x. u)t\mapsto u[t/x]\]
 \[ (\lambda \alpha. u)m\mapsto u[m/\alpha]\]
  \[ \proj{i}{\pair{u_0}{u_1}}\mapsto u_i, \mbox{ for i=0,1}\]
 \[\case{\inj{i}{u}}{x_{1}.t_{1}}{x_{2}.t_{2}}\mapsto t_{i}[u/x_{i}], \mbox{ for i=0,1} \]
 \[\dest{\enc{m}{u}}{x}{v} \mapsto v[m/\alpha][u/x], \mbox{ for each term $m$ of $\mathcal{L}$} \]
 \[\D\, (\lambda \alpha\, \inj{0}{u}) \mapsto \inj{0}{\lambda \alpha\, u}\]
  \[\D\, (\lambda \alpha\, \inj{1}{u}) \mapsto \inj{1}{u[\dum/\alpha]}, \mbox{ where $\dum$ is a fixed constant of $\mathcal{L}$} \]  
\end{description}
\dlinea
\caption{Reduction Rules for $\CD$}\label{fig:red-ilmp}
\end{figure*}

$\CD$ with the reduction rules in figure \cref{fig:red-ilmp} enjoys the Subject Reduction Theorem.
\begin{theorem}[Subject Reduction]\label{subjectred}
If $ t : C$ and $t \mapsto u$, then $ u : C$. Moreover, the free proof-term variables of $u$ are among those of $t$.
\end{theorem}

\begin{proof}
Every reduction rule for intuitionistic logic satisfies the statement: see \cite{Sorensen}.  We thus have only to check that the reductions associated to the constants $\D$ preserve the type as well.

The first reduction is
$$\D\, (\lambda \alpha\, \inj{0}{u}) \mapsto \inj{0}{\lambda \alpha\, u}$$
where the type derivation of $\D\, (\lambda \alpha\, \inj{0}{u})$ has the following shape:
\begin{prooftree}
\AxiomC{$ \D: \forall \alpha (A\lor B)\rightarrow \forall \alpha A \lor B$}
\AxiomC{$ u: A$}
\UnaryInfC{$ \inj{0}{u}: A\lor B$}
\UnaryInfC{$ \lambda \alpha\, \inj{0}{u}: \forall \alpha (A\lor B)$}
\BinaryInfC{$ \D\, (\lambda \alpha\, \inj{0}{u}): \forall \alpha A\lor B$}
\end{prooftree}
with $\alpha$ not occurring in the types of the free variables of $\inj{0}{u}$. Therefore, the following derivation is correct 
\begin{prooftree}
\AxiomC{$ u: A$}
\UnaryInfC{$ \lambda\alpha\,u : \forall\alpha A$}
\UnaryInfC{$ \inj{0}{\lambda \alpha\, u}: \forall \alpha A\lor B$}
\end{prooftree}
which is what we wanted to show.

The second reduction is
$$\D\, (\lambda \alpha\, \inj{1}{u}) \mapsto \inj{1}{u[\dum/\alpha]}$$
where the type derivation of $\D\, (\lambda \alpha\, \inj{1}{u})$ has the following shape:
\begin{prooftree}
\AxiomC{$ \D: \forall \alpha (A\lor B)\rightarrow \forall \alpha A \lor B$}
\AxiomC{$u: B$}
\UnaryInfC{$ \inj{1}{u}: A\lor B$}
\UnaryInfC{$ \lambda \alpha\, \inj{1}{u}: \forall \alpha (A\lor B)$}
\BinaryInfC{$\D\, (\lambda \alpha\, \inj{1}{u}): \forall \alpha A\lor B$}
\end{prooftree}
where $\alpha$ does not occur neither in the types of the free variables of $\inj{1}{u}$ nor in $B$. Now, for every $ t: C$, term $m\in\mathcal{L}$ and first-order variable $\beta$,  one has  $ t[m/\beta]: C[m/\beta]$ (see \cite{Sorensen}). Therefore, the following derivation is correct 
\begin{prooftree}
\AxiomC{$ u[\dum/\alpha]: B$}
\UnaryInfC{$ \inj{1}{u[\dum/\alpha]}:  \forall \alpha A\lor B$}
\end{prooftree}
which is what we wanted to show.
  
\end{proof}

 In order to derive the constructiveness of $\CD$, we shall just have to inspect the normal forms of proof terms. Our main argument, in particular, will use the following well-known syntactic characterization of the shape of proof terms.
\begin{proposition}[Head of a Proof Term]
  \label{theorem:head-form}
  Every proof-term of $\CD$ is of the form 
\[\lambda z_1 \dots \lambda z_n.\, r\, u_1 \dots u_k \]
where 
\begin{itemize}
\item $r$ is either a variable or a constant or a term corresponding to an introduction rule: $\lambda x\, t$, $\lambda \alpha\, t$, $\pair{t_1}{t_2}$, $\inj{i}{t}$, $\enc{m}{t}$
\item $u_1, \dots u_k$ are either proof terms, first order terms, or one of the following expressions corresponding to elimination rules: $\proj{i}{}$, $\case{}{x.w_1}{y.w_2}$, $\dest{}{x}{t}$.
\end{itemize}
\end{proposition}
\begin{proof}
Standard. 
\end{proof}
We can now prove that $\CD$ is constructive.

\begin{theorem}[Disjunction and Witness Property for  $\CD$]
\label{thm:construct-il-mp}\mbox{}
\begin{enumerate}
\item
If $\CD \vdash t: \exists \alpha \, A$, and $t$ is in normal form, then $t=\enc{m}{u}$ and $\CD \vdash u: A[m/\alpha]$.
\item If $\CD \vdash t: A \lor B$ and $t$ is in normal form, then either $t=\inj{0}{u}$ and $\CD \vdash u: A$ or $t=\inj{1}{u}$ and $\CD\vdash u: B$.
\item If $\CD \vdash t: \forall \alpha \, A$, and $t$ is in normal form, then $t=\lambda \alpha\, {u}$ and $\CD \vdash u: A$.
\end{enumerate}
\end{theorem}
\begin{proof}\mbox{}
We prove 1, 2, 3 simultaneously by induction on $t$. If $t=\lambda \alpha\, u$, we are done; assume then it is not. By \Cref{theorem:head-form}, $t=r\, u_1\dots u_k$. Let us explore the possible forms of $r$.
  \begin{itemize}
  \item Since $t$ is closed, $r$ cannot be a variable.
  \item For the sake of contradiction, suppose $r=\D$. Then $\CD \vdash u_{1}: \forall\alpha (A\lor B)$, with $\alpha$ not in $B$. By induction hypotheses 2 and 3, we have $u_{1}=\lambda \alpha\, \inj{i}{u}$, with $i\in\{0,1\}$. Thus, 
  $t=\D\, (\lambda \alpha\, \inj{i}{u})\, u_{2}\ldots u_{k}$, which is not a normal form.
  
  \item  It cannot  be   $r=\efq{P}{}$ either, otherwise  $\CD \vdash u_1 : \bot$, which is impossible, by consistency of the logic.
  \item The only remaining possibility is that $r$ is one among $\lambda x\, t$, $\lambda \alpha\, t$, $\pair{t_1}{t_2}$, $\inj{i}{t}$, $\enc{m}{t}$. In this case, $k$ must be 0 as otherwise we would have a redex. This means that $t=r$. 
Therefore,  if  $\CD\vdash t: \exists \alpha\, A$, then $t=\enc{m}{u}$ with $\CD \vdash u : A(m)$. If $\CD\vdash t: A \lor B$, then either $t=\inj{0}{u}$ and $\CD \vdash u: A$ or $t=\inj{1}{u}$ and $\CD\vdash u: B$. 
  \end{itemize}
\end{proof}

In order to derive strong normalization for $\CD$, we are now going to define a simple translation mapping proof terms of  $\CD$ into terms of  the intuitionistic logic calculus $\IL$ extended with a constant for falsity. The trick allows us to use  in the translation some dummy terms that will never be involved in the reduction rules, but help the translation to preserve the type. 
\begin{definition}[$\ILB$, Translation of $\CD$ into $\ILB$]\mbox{}
\begin{enumerate}
\item
Let $\ILB$ the system obtained from $\CD$ by removing the constant domain axiom and adding a rule
$$\mathsf{F}: \bot$$ 
where $\mathsf{F}$ is a constant symbol. Moreover, for every $A$, there is a proof term of type $\bot\rightarrow A$, thus there is some dummy closed term of type $A$ in $\ILB$: we denote it with $\mathsf{d}^{A}$.
\item We define a translation $\trans{\_}: \CD\rightarrow \ILB$, leaving types unchanged. 
For terms $\D^{I}: I$, where $I=\forall\alpha (A\lor B)\rightarrow \forall \alpha A \lor B$ is an instance of the constant domain axiom, we define:
$$\trans{(\D^{I})}:= \lambda f^{\footnotesize{\forall\alpha (A\lor B)}}\, f\,\dum\, [z^{\footnotesize{A[\dum/\alpha]}}.\, \inj{0}{\lambda \alpha\, f\, \alpha\, [x^{\footnotesize{A}}.\, x, y^{\footnotesize{B}}.\, \mathsf{d}^{\footnotesize{A}}]},\, z^{\footnotesize{B}}.\,\inj{1}{z}]$$
and by construction $\ILB\vdash \trans{(\D^{I})}: I$. For all other proof terms $t$ of  $\CD$, we set
 $\trans{t}$ as the term of $\ILB$ obtained from $t$ by replacing all its constants $\D^{I}$ with $\trans{(\D^{I})}$.
 \end{enumerate}
\end{definition}

As a matter of fact, each reduction step between $\CD$ terms corresponds to \emph{at least} a step between their translations. With $\mapsto^{+}$ we denote the transitive closure of $\mapsto$. 

\begin{proposition}[Preservation of the Reduction Relation]\label{proposition-preservation}
Let $v$ be any term of $\CD$. Then 
$$v\mapsto w\Rightarrow \trans{v}\mapsto^{+}\trans{w}$$
\end{proposition}
\begin{proof} It is sufficient to prove the proposition when $v$ is a redex $r$. We have several possibilities:
\begin{enumerate}
\item $r=(\lambda x u)t\mapsto u[t/x]$. 
We verify indeed that $$ \trans{((\lambda x u)t)}=(\lambda x \trans{u})\trans{t}\mapsto \trans{u}[\trans{t}/x]= \trans{u[t/x]}$$
\item $r=\pair{u_{0}}{u_{1}}\pi_{i}\mapsto u_{i}$. 
We verify indeed that 
$$\trans{(\pair{u_{0}}{u_{1}}\pi_{i})}=\pair{\trans{u}_{0}}{\trans{u}_{1}}\pi_{i}\mapsto \trans{u}_{i}$$
\item $r=\inj{i}{u}\, [x_{1}. v_{1}, x_{2}, v_{2}]\mapsto v_{i}[u/x_{i}]$. We verify indeed that 
$$\trans{(\inj{i}{u}\, [x_{1}. v_{1}, x_{2}, v_{2}])}=\inj{i}{\trans{u}}\, [x_{1}. \trans{v_{1}}, x_{2}, \trans{v_{2}}]\mapsto \trans{v_{i}}[\trans{u}/x_{i}]=\trans{(v_{i}[u/x_{i}])}$$
\item The other intuitionistic reductions are analogous. 

\item $r=\D\, (\lambda \alpha\, \inj{0}{u})\mapsto \inj{0}{\lambda \alpha \, u}$. We verify indeed that 
\[\begin{aligned}\trans{(\D\, (\lambda \alpha\, \inj{0}{u}))}=&\left(\lambda f\, f\,\dum\, [z.\, \inj{0}{\lambda \alpha\, f\, \alpha\, [x.\, x, y.\, \mathsf{d}]},\, z.\,\inj{1}{z}]\right)\, (\lambda \alpha\, \inj{0}{\trans{u}})\\
&\mapsto (\lambda \alpha\, \inj{0}{\trans{u}})\,\dum\, [z.\, \inj{0}{\lambda \alpha\, (\lambda \alpha\, \inj{0}{\trans{u}})\, \alpha\, [x.\, x, y.\, \mathsf{d}]},\, z.\,\inj{1}{z}]\\
&\mapsto \inj{0}{\trans{u}[\dum/\alpha]}\, [z.\, \inj{0}{\lambda \alpha\, (\lambda \alpha\, \inj{0}{\trans{u}})\, \alpha\, [x.\, x, y.\, \mathsf{d}]},\, z.\,\inj{1}{z}]\\
&\mapsto \inj{0}{\lambda \alpha\, (\lambda \alpha\, \inj{0}{\trans{u}})\, \alpha\, [x.\, x, y.\, \mathsf{d}]}\\
&\mapsto \inj{0}{\lambda \alpha\,  \inj{0}{\trans{u}}\, [x.\, x, y.\, \mathsf{d}]}\\
&= \inj{0}{\lambda \alpha\,\trans{u}}= \trans{(\inj{0}{\lambda \alpha\, u})}
\end{aligned}
\]
\item $r=\D\, (\lambda \alpha\, \inj{1}{u})\mapsto \inj{1}{u[\dum/\alpha]}$. We verify indeed that 
\[\begin{aligned}\trans{(\D\, (\lambda \alpha\, \inj{1}{u}))}&=\left(\lambda f\, f\,\dum\, [z.\, \inj{0}{\lambda \alpha\, f\, \alpha\, [x.\, x, y.\, \mathsf{d}]},\, z.\,\inj{1}{z}]\right)\, (\lambda \alpha\, \inj{1}{\trans{u}})\\
&\mapsto (\lambda \alpha\, \inj{1}{\trans{u}})\,\dum\, [z.\, \inj{0}{\lambda \alpha\, (\lambda \alpha\, \inj{1}{\trans{u}})\, \alpha\, [x.\, x, y.\, \mathsf{d}]},\, z.\,\inj{1}{z}]\\
&\mapsto \inj{1}{\trans{u}[\dum/\alpha]}\, [z.\, \inj{0}{\lambda \alpha\, (\lambda \alpha\, \inj{0}{\trans{u}})\, \alpha\, [x.\, x, y.\, \mathsf{d}]},\, z.\,\inj{1}{z}]\\
&\mapsto \inj{1}{\trans{u}[\dum/\alpha]}\\
&=\trans{(\inj{1}{u[\dum/\alpha]})}
\end{aligned}
\]

\end{enumerate}
\end{proof}
Strong normalization theorem for $\CD$ easily follows from Proposition \ref{proposition-preservation}.
\begin{theorem}[Strong Normalization  for $\CD$]
Suppose $t: A$ in $\CD$. Then $t$ is strongly normalizable with respect to the relation $\mapsto$.
\end{theorem}
\begin{proof} By Proposition \ref{proposition-preservation}, any infinite reduction $t=t_{1}, t_{2}, \ldots, t_{n}, \ldots $ in System $\CD$ gives rise to an infinite reduction $\trans{t}=\trans{t}_{1}, \trans{t}_{2}, \ldots, \trans{t}_{n}, \ldots$ in  $\ILB$.  By the strong normalization Theorem for $\IL$ (see \cite{Sorensen}), infinite reductions of the  latter kind cannot occur; thus neither of the former. 
\end{proof}

\section{The Constant Domain Axiom in Arithmetic}
\label{sec:arithmetic}

 Many mathematical axioms and concepts preserve the computational properties of intuitionistic logic, when added on top of it. Thus not only intuitionistic logic  embodies constructive logical reasoning, it also serves as a foundation to constructive mathematics. Since we showed that the intuitionistic logic of constant domains is constructive too, the natural question is whether it can be a basis for constructive mathematics as well. 
 
 The answer is no. When extended with the standard axiom scheme for induction on natural numbers, $\CD$ does not have the disjunction property and, as a consequence, the existential property does not hold either. Showing it is the goal of this section.
 
 We start by introducing the system $\HA+\CD$ which is just the standard Heyting intuitionistic Arithmetic (see \cite{Sorensen}), augmented with the constant domain axiom. 
 
 \begin{definition}[Language of $\HA + \CD$]\label{definition-languagear}
  The language $\mathcal{L}$ of $\HA + \CD$ is defined as follows.
  \begin{enumerate}
  \item The terms of $\mathcal{L}$ are inductively defined as either variables $\alpha, \beta,\ldots$ or $\dum$ or $\suc(t)$ with $t\in\mathcal{L}$. A numeral is a term of the form $\suc\ldots \suc \dum$. 
  \item There is one symbol $\mathcal{P}$ for every primitive recursive relation over $\mathbb{N}$; with $\mathcal{P}^{\bot}$ we denote the symbol for the complement of the relation denoted by $\mathcal{P}$. The atomic formulas of $\mathcal{L}$ are all the expressions of the form $\mathcal{P}(t_{1}, \ldots, t_{n})$ such that $t_{1}, \ldots, t_{n}$ are terms of $\mathcal{L}$ and $n$ is the arity of $\mathcal{P}$. Atomic formulas will also be denoted as $\prop{P}, \mathsf{Q}, \prop{P_i}, \ldots$ and $\mathcal{P}(t_{1}, \ldots, t_{n})^{\bot}:=\mathcal{P}^{\bot}(t_{1}, \ldots, t_{n})$. 
\item The formulas of $\mathcal{L}$ are built from atomic formulas of $\mathcal{L}$ by the connectives $\lor,\land,\rightarrow, \forall,\exists$ as usual.\\
\end{enumerate}
\end{definition}

 \begin{definition}[$\HA+\CD$]
We define $\HA+\CD$ to be the system obtained from $\CD$ fixing $\Language$ as formula language and adding the induction rule:

$$\begin{array}{c} \Gamma\vdash u: A(\dum)\ \ \ \ \  \Gamma\vdash v:\forall \alpha.\, A(\alpha)\rightarrow A(\suc(\alpha))\\ \hline \Gamma\vdash \rec\, u\,v\,m : A[m/\alpha] \end{array}$$
together with defining axioms for primitive recursive relation.
\end{definition}

\begin{theorem}
$\HA+\CD$ does not have the disjunction and the witness properties.
\end{theorem}
\begin{proof}
Suppose for the sake of contradiction that $\HA+\CD$ has the disjunction property. Let us consider the Kleene primitive recursive predicate $\mathcal{T}(x, y, z)$, which holds if and only if $z$ codes a terminating computation for the Turing machine $x$ on input $y$ \cite{Odi}. Then
$$\HA+\CD \vdash\forall \alpha\,( \lnot\mathcal{T}(x, y, \alpha) \lor \exists \beta\, \mathcal{T}(x, y, \beta) )$$ 
because the excluded middle holds for recursive relations and namely for $\mathcal{T}$. 
Since the formula $$\forall\alpha\, (\lnot\mathcal{T}(x, y, \alpha) \lor \exists \beta\, \mathcal{T}(x, y, \beta)) \rightarrow \forall\alpha\, \lnot\mathcal{T}(x, y, \alpha) \lor \exists \beta\, \mathcal{T}(x, y, \beta) $$
is an instance of the constant domain axiom, we obtain
$$\CD\vdash \forall\alpha\, \lnot\mathcal{T}(x, y, \alpha) \lor \exists \beta\, \mathcal{T}(x, y, \beta)$$
Since by hypothesis $\CD$  has the disjunction property, for every pair of numerals $n, m$, either $\CD\vdash \forall\alpha\,\lnot\mathcal{T}(n, m, \alpha)$ or $\CD\vdash \exists \beta\, \mathcal{T}(n, m, \beta)$. We thus have at disposal a recursive procedure for solving the Halting Problem, which is a contradiction. 

Finally, disjunctions $A\lor B$ can be coded as $\exists \alpha.\, (\alpha=\dum \rightarrow  A) \land (\alpha=\suc(\dum) \rightarrow  B)$ and therefore if $\HA+\CD$ had the witness property, it would have the disjunction property, again a contradiction.
\end{proof}

\end{document}